\documentclass[preprint,3p]{elsarticle}
\biboptions{sort&compress}

\usepackage[pdftex,colorlinks]{hyperref}
 \usepackage[pagewise,mathlines]{lineno}

\journal{Arxiv}

 \usepackage{tikz}
 \usetikzlibrary{shapes,decorations}

 \usepackage{environ}

\definecolor{mycolor}{RGB}{255,251,204}
\tikzstyle{mybox} = [draw=yellow, very thick,
rectangle, rounded corners, inner sep=5pt, inner ysep=10pt]
\tikzstyle{fancytitle} =[fill=white, text=red]

\NewEnviron{blc}{\begin{tikzpicture}
	\node [mybox] (box){%
		\begin{minipage}{0.95\textwidth}
		\BODY
		\end{minipage}
	};
	\end{tikzpicture}
}

\usepackage{amsfonts}
\usepackage{amssymb}
\usepackage[utf8]{inputenc}
\usepackage{amsmath}
\usepackage{amsthm}
\usepackage{graphicx}%
\usepackage{enumerate}
\usepackage{xcolor}
\usepackage{color,soul}
\usepackage{tikz}
\usetikzlibrary{patterns}
\usepackage{mathtools}

\newtheorem{thm}{Theorem}[section]

\newtheorem{lem}[thm]{Lemma}
\newtheorem{prop}[thm]{Proposition}

\newcommand{\nh}[1]{\|#1\|_{H^1_\alpha}}

\newcommand{\ha}{H^1_\alpha}

\theoremstyle{definition}
\newtheorem{defn}[thm]{Definition}

\theoremstyle{remark}
\newtheorem{rem}[thm]{Remark}

\numberwithin{equation}{section}

\newcommand{\ep}{\varepsilon}

\newcommand{\R}{\mathbb{R}}				      

\newcommand{\dps}{\displaystyle}	
\newcommand{\intq}{\int_0^T\int_0^1}	
\newcommand{\intw}{\int_0^T\int_{1-\ep}^1}	
	
\newcommand{\into}{\int_0^1}	

\newcommand{\dd}{\,dxdt}

\newcommand{\n}[2]{\|#1\|_{{#2}}}

\usepackage[textsize=scriptsize,textwidth=15mm]{todonotes}
\usepackage{marginnote}
\makeatletter
\renewcommand{\@todonotes@drawMarginNoteWithLine}{%
	\begin{tikzpicture}[remember picture, overlay, baseline=-0.75ex]%
	\node [coordinate] (inText) {};%
	\end{tikzpicture}%
	\marginnote[{
		\@todonotes@drawMarginNote%
		\@todonotes@drawLineToLeftMargin%
	}]{
		\@todonotes@drawMarginNote%
		\@todonotes@drawLineToRightMargin%
	}%
}
\makeatother

\allowdisplaybreaks[1] 

\begin{document}

\begin{frontmatter}

\title{\textsc{Regularity results for degenerate wave equations in a neighborhood of the boundary}}

\author[UFCG]{B. S. V. Ara\'ujo}
\ead{bsergio@mat.ufcg.edu.br}

\author[RCN]{R. Demarque\corref{mycorrespondingauthor}}
\cortext[mycorrespondingauthor]{Corresponding author}
\ead{reginaldo@id.uff.br}

\author[GAN]{L. Viana}
\ead{luizviana@id.uff.br}

\address[UFCG]{Unidade Acadêmica de Matemática, Universidade Federal de Campina Grande, Campina Grande, PB, Brazil}
\address[RCN]{Departamento de Ciências da Natureza,
	Universidade Federal Fluminense,
	Rio das Ostras, RJ, Brazil}
\address[GAN]{Departamento de Análise,
	Universidade Federal Fluminense,
	Niter\'{o}i, RJ, Brazil}

\begin{abstract}
In this paper we establish some regularity results concerning the behavior of weak solutions and very weak solutions of the degenerate wave equation near the boundary. For the nondegenerate case, the correponding results were originally obtained by Fabre and Puel (J. of Diff. Eq. 106, 1993). This kind of results is closely related to the exact boundary controllability for the wave equation as the limit of internal controllability.
\end{abstract}


\end{frontmatter}

\section{Introduction and Statements of the main results}\label{intro}

In this work we are interested in studying the behavior, near the boundary point $x=1$, of the weak and very weak solutions of the following degenerate wave equation: 
\begin{equation}
	\begin{cases}
	u_{tt}- \displaystyle \left(x^\alpha u_{x} \right)_x=f, & (t,x)\in Q,\\
	u(t,1)=0,& \text{ in } (0,T),\\
	\begin{cases}
	u (t,0)=0, &\text{if } \alpha \in (0,1),\\
	\text{or}& \\
	(x^\alpha u_{ x})(t,0)=0,&  \text{if } \alpha \in [1,2),
	\end{cases} &  t\in (0,T), \\
	u (0,x)=u_0(x) \text{ and } u_{t} (0,x)=u_1(x) & x\in  (0,1),
	\end{cases}
	\label{pb1}
	\end{equation}
where $T>0$,  {$Q=(0,T)\times (0,1)$}, $\alpha\in(0,2)$ and the data $(f,u_0,u_1)$ belongs to spaces that will determine the regularity of the solution. To develop this study, the $L^2$ norm of the solution will be analyzed in an $\varepsilon$-neighborhood of the boundary point $x=1$.

For the nondegenerate wave equation, an analogous investigation has been considered by Fabre and Puel in \cite{fabre1993behavior}. Their results have played a key role in \cite{fabre1992exact}, where an exact boundary controllability is achieved as the limit of a sequence of internal controllability problems, set in $\varepsilon$- neighborhoods of the boundary, as $\varepsilon \to 0$. Here, we are supposed to remark that the previous work arises from \cite{zuazua1988controlabilite}, where Zuazua has used Lion's Hilbert uniqueness method (HUM) to assure the exact controllability for the nondegenerate wave equation when the distributed control acts on an $\ep$-neighborhood $\omega_\ep$ of $\Gamma_0 \subset \Gamma$, where $\Gamma$ denotes the boundary of the domain. So that, in \cite{fabre1992exact}, a passage to the limit procedure can be established, by studying the convergence of solutions of 
\begin{equation}\label{sing-eq}
\psi_{\varepsilon tt}-\Delta \psi_\ep=\frac{1}{\ep^3}{\varphi}_\ep\chi_{\omega_\ep\times (0,T)}, \text{ in } Q,
\end{equation}
where there exists a singular right-hand side with respect to $\ep$ and   $\frac{1}{\ep^3}{\varphi}_\ep$ is the control given by HUM as a solution of an adjoint problem.

Recently, in \cite{chaves2020boundary}, Chaves-Silva, Puel and Santos have extended the research contained in \cite{fabre1992exact} for the nondegenerate heat equation context. Shortly afterwards, in \cite{araujo2022boundary}, we have proposed a natural continuation for the subject of \cite{chaves2020boundary} and \cite{fabre1992exact}, taking into consideration the one-dimensional degenerate heat equation. At this point, asking ourselves if those results presented in \cite{araujo2022boundary,chaves2020boundary,fabre1992exact} can be proved for the one-dimensional degenerate wave equation seems a realistic sequel.

Concerning the nondegenerate heat and wave equations, we would also contrast the following facts: even though \cite{chaves2020boundary} and \cite{fabre1992exact} communicate similar theorems, their obtainment does not come from the same techniques. This is really expected, since, in the whole Control Theory, parabolic and hyperbolic equations are usually dealt with different approaches. Both problems are solved using a bound for the $L^2$-norms of weak solutions in $\varepsilon$-neighborhoods of the boundary. However, in \cite{chaves2020boundary}, the desired bound is achieved from specific weighted Carleman estimates, while, in \cite{fabre1992exact}, it can be found relying on the regularity results presented in \cite{fabre1993behavior}. In fact, while passing to the limit,  such a regularity results are crucial in order to check the continuity of a certain functional. As pointed out by Fabre in \cite{fabre1993behavior}, equations like \eqref{sing-eq} can be stated regardless the exact controllability context, which means that the mentioned regularity results can be useful in other more general situations. Having in mind this initial discussion, before trying to extend \cite{araujo2022boundary} for the degenerate wave equation context, our analysis starts from this current paper, where we will obtain  regularity theorems  which are analogous to that ones proved in \cite{fabre1993behavior}.   

Next, before stating our main results, let us present some important functional spaces, introduced in \cite{alabau2006carleman}.

\begin{defn} [Weighted Sobolev spaces]
Consider $\alpha \in (0,1)$, for the \textbf{weakly degenerate case} (WDC), or $\alpha \in [1,2)$, for the \textbf{strongly degenerate case} (SDC).

\begin{itemize}
	\item[(I)] For the (WDC), we set 
	\begin{equation*}
	H_{\alpha}^1:= \{  u\in L^2(0,1);\ u\mbox{ is absolutely continuous in } [0,1],
	x^{\alpha /2}u_x\in L^2(0,1) \mbox{ and } u(1)=u(0)=0\},
	\end{equation*}
	equipped with the natural norm
\[ \|u\|_{H_{\alpha}^1}:=\left( \|u\|_{L^2(0,1)}^2+\| x^{\alpha /2} u_x\|_{L^2(0,1)}^2 \right) ^{1/2} ;\]
\item[(II)] For the (SDC),
\begin{equation*}
H_{\alpha}^1:= \{  u\in L^2(0,1);\ u\text{ is  {locally} absolutely continuous in } (0,1],
x^{\alpha /2} u_x\in L^2(0,1) \mbox{ and } u(1)=0\},
\end{equation*}  
and the norm keeps the same;

\item[(III)] In both situations, the (WDC) and the (SDC), 
\[
H_{\alpha}^2:= \{  u\in H_{\alpha}^1;\ x^{\alpha /2} u_x\in H^1(0,1) \}
\]
with the norm
$\|u\|_{H_{\alpha}^2}:=\left( \|u\|_{H_{\alpha}^1}^2+\|(x^{\alpha /2} u_x)_x\|_{L^2(0,1)}^2 \right) ^{1/2}$. 
\end{itemize}
\end{defn}	

Another important space in this context is $H_\alpha^{-1}=(H_\alpha^1)'$ (the dual space of $H_\alpha^1$). For  $u\in H_\alpha^{-1}$, from Lax-Milgram Theorem, there exists a unique  $\tilde{u}\in H_\alpha^{1}$ such that 
\[\langle u,v\rangle_{H_\alpha^{-1}}=\int_0^1x^\alpha \tilde{u}_xv_x\,dx \ \ \forall v\in H_\alpha^1.\]
Hence, $H_\alpha^{-1}$ is a Hilbert space equipped with the inner product $$(u,v)_{H_\alpha^{-1}}=\int_0^1x^\alpha\tilde{u}_x\tilde{v}_x\,dx.$$

Next, let us specify which kind of solution for \eqref{pb1} we will deal with.

\begin{defn}\label{weak} Given  $f\in L^1(0,T;L^2(0,1))$ and $(u_0,u_1)\in H^1_\alpha\times L^2(0,1)$, we say that 
\[
\displaystyle u\in C([0,T];H^1_\alpha)\cap C^1([0,T];L^2(0,1))
\]
is a \textbf{weak solution} of the system \eqref{pb1} if the following properties hold:
\begin{itemize}
\item[(a)] $u(0,x)=u_0(x)$ for all $x\in (0,1)$; 
\item[(b)] 
\begin{equation*}
    \intq (-u_t\varphi_t+x^\alpha u_x\varphi_x)\,dxdt -\into u_1\varphi(0,x)\,dx=\intq f\varphi\,dxdt,
\end{equation*}
for all $\varphi\in L^2(0,T;H^1_\alpha)$ satisfying $\varphi_t\in L^2(Q)$ and $\varphi(T,\cdot)=0$.
\end{itemize}
\end{defn}
At this moment, we are to state our two main results:
\begin{thm}\label{th2.1}
Given $0<\ep_0<1$, there exists $C>0$ such that, for all $(u_0,u_1)\in H_\alpha^1\times L^2(0,1)$ and $f\in L^1(0,T;L^2(0,1))$, if $u$ is a weak solution to \eqref{pb1}, then
\begin{equation*}
    \frac{1}{\ep^3}\intw |u(t,x)|^2 \dd  \leq C\left(\n{f}{L^1(0,T;L^2(0,1))}^2+\n{u_0}{H^1_\alpha}^2+\n{u_1}{L^2(0,1)}^2\right),\ \forall \ep \in (0,\ep _0 ],
\end{equation*}
where $C$ only depends on $\ep_0$, $\alpha$ and $T$.
\end{thm}

\begin{thm}\label{th2.3}
Given $0<\ep_0<1$, there exists $C>0$ such that, for all $(u_0,u_1)\in H_\alpha^1\times L^2(0,1)$ and $f\in L^1(0,T;L^2(0,1))$, if $u$ is a weak solution to \eqref{pb1}, then
\begin{equation*}
    \frac{1}{\ep}\intw |u_x(t,x)|^2 x^\alpha\dd  \leq C\left(\n{f}{L^1(0,T;L^2(0,1))}^2+\n{u_0}{H^1_\alpha}^2+\n{u_1}{L^2(0,1)}^2\right),\ \forall \ep \in (0,\ep _0 ], 
\end{equation*}
where $C$ only depends on $\ep_0$, $\alpha$ and $T$.
\end{thm}

We know that a weak solution $u$ of problem \eqref{pb1} satisfies $u_x(\cdot, 1)\in L^2(0,T)$, see Proposition \ref{prop-trace}. However, if we take $(u_0, u_1)\in L^2 (0,1) \times H^{-1}_\alpha$ (instead of taking it in $H^1_\alpha \times L^2 (0,1)$), we need to consider solutions $\tilde{u}=\tilde{u}(t,x)$ in a very weak sense. In this case, we have no information about the regularity of $\tilde{u}_x (\cdot ,1)$.

In our next result, as in the nondegenerate case, we will provide this kind of  regularity for a function $\varphi$ that is the limit of a sequence $(\varphi_\ep)$  of very weak solutions, under a condition of the $L^2$ norms of $(\varphi_\ep)$ near the boundary. In the following, we will precise this result.

Given $(g,z^0,z^1)\in L^1(0,T;L^2(0,1))\times L^2(0,1)\times H^{-1}_\alpha$, let us consider the problem
\begin{equation}
	\begin{cases}
	z_{tt}- \displaystyle \left(x^\alpha z_{x} \right)_x=g, & (t,x)\in Q,\\
	z(t,1)=0,& \text{ in } (0,T),\\
	\begin{cases}
	z(t,0)=0, &\text{if } \alpha \in (0,1),\\
	\text{or}& \\
	(x^\alpha z_{ x})(t,0)=0,&  \text{if } \alpha \in [1,2),
	\end{cases} &  t\in (0,T), \\
	z(0)=z^0 \text{ and } z_{t} (0)=z^1.
	\end{cases}
	\label{pb2}
	\end{equation}

In the following, we will present the definition of solutions for this problems that can be found in \cite[page 47]{lions1988controlabilite}.
\begin{defn}\label{trans}
 Given  $g\in L^1(0,T;L^2(0,1))$ and $(z^0,z^1)\in L^2(0,1)\times H^{-1}_\alpha $, we say  $z\in L^\infty(0,T;L^2(0,1))$  is a \textbf{very weak solution}  
(or a \textbf{solution by transposition}) of \eqref{pb2} if, for each $F\in \mathcal{D}(Q)$, 
\begin{equation*}
    \intq uF\,dxdt=-(z_0,\theta'(0))+\langle z_1,\theta(0)\rangle+\intq g\theta \,dxdt,
\end{equation*}
where $\theta = \theta (t,x)$ solves
\begin{equation*}
	\begin{cases}
	\theta_{tt}- \displaystyle \left(x^\alpha \theta_{x} \right)_x=F, & (t,x)\in Q,\\
	\theta(t,1)=0,& \text{ in } (0,T),\\
	\begin{cases}
	\theta (t,0)=0, &\text{if } \alpha \in (0,1),\\
	\text{or}& \\
	(x^\alpha \theta_{ x})(t,0)=0,&  \text{if } \alpha \in [1,2),
	\end{cases} &  t\in (0,T), \\
	  {\theta (T,x)=\theta_{t} (T,x)=0} & x\in  (0,1).
	\end{cases}
		\end{equation*}
   {Above, $\mathcal D (Q)$ denotes the real vector space of all smooth and compactly supported functions defined on $Q$.}
\end{defn}
For the well-posedness of \eqref{pb2}, in the sense of transposition, see Proposition \ref{well-pos-trans}. At this point, it is important to mention that one of the main differences between weak and very weak solutions are the following: the weak solution has the so called ``hidden regularity" (Proposition \ref{prop-trace}), while this property is not true for the very weak solution.

Our next result plays like a reciprocal of Theorem \ref{th2.3}, but, in fact, it is not really. Indeed, Theorem \ref{th3.1} concerns about a sequence of very weak solutions $\varphi_n$ that converges (in some weak sense) to a very weak solution $\varphi$. If this sequence satisfies a bound near the boundary, similar to that one got in the previous Theorem \ref{th2.3}, then the very weak solution $\varphi$ has a  hidden regularity.

Let us consider a family of functions $(h_\ep,\varphi^0_\ep,\varphi^1_\ep)\in L^1(0,T:L^2(\Omega))\times L^2(\Omega)\times H^{-1}_\alpha$ such that

\begin{equation*}
\begin{aligned}
   & h_\ep \rightharpoonup h && \text{ in } L^1(0,T:L^2(\Omega)),\\
    & \varphi_\ep^0 \rightharpoonup \varphi^0 &&  \text{ in } L^2(\Omega),\\
    & \varphi_\ep^1 \rightharpoonup \varphi^1 && \text{ in } H^{-1}_\alpha,
\end{aligned}
\end{equation*}
and let $\varphi_\ep$ be the solution by transposition of problem \eqref{pb2} with $(g,z^0,z^1)=(h_\ep,\varphi_\ep^0,\varphi_\ep^1)$.
Then $\varphi_\ep\in C^0([0,T];L^2(0,1))\cap C^1([0,T];H^{-1}_\alpha)$ and $\varphi_\ep \stackrel{\ast}{\rightharpoonup} \varphi$ in $L^\infty(0,T;L^2(0,1))$, where $\varphi$ is the solution by transposition of \eqref{pb2} with $(g,z^0,z^1)=(h,\varphi^0,\varphi^1)$.

\begin{thm}\label{th3.1}
Let $(h_\ep,\varphi^0_\ep,\varphi^1_\ep)\in L^1(0,T:L^2(\Omega))\times L^2(0,1)\times H^{-1}_\alpha$ be a family of functions as described above. If 
\begin{equation}\label{3.4}
    \frac{1}{\ep^3}\intw |\varphi_\ep(t,x)|^2 \dd \leq C,
\end{equation}
where $C$ does not depend on $\ep$, then $\varphi_x(\cdot,1)\in L^2(0,T)$ and 
\begin{equation}\label{liminf}
    \frac{1}{3}\n{\varphi_x(\cdot,1)}{L^2(0,T)}^2
    \leq \liminf_{\ep\to 0^+}\left(\frac{1}{\ep^3}\intw |\varphi_\ep (t,x)|^2\dd \right)
\end{equation}
\end{thm}

\begin{rem}
In this paper, we consider weakly and strongly degenerate wave equations, with Dirichlet and Neumann boundary conditions, respectively. However, we should emphasize that these two different situations can be treated simultaneously. Indeed, each integration by parts provides boundary terms which vanish in both cases at $x=0$, while the boundary condition is the same at $x=1$.
\end{rem}


%


%

The remainder of this paper is organized as follows: in Section \ref{sec-pre} we present some well posedness results concerning the weak solution and the very weak solution of \eqref{pb1}. The proof of the Theorems \ref{th2.1} and \ref{th2.3} are discussed in Section \ref{sec-weak}. The paper ends with the Section \ref{sec-very-weak}, where we present the proof of Theorem \ref{th3.1}.

\section{Preliminaries}\label{sec-pre}

As expected from a paper like this, we start presenting some well-posedness results for the hiperbolic degenerate equation \eqref{pb1}. Most of these results are well known and can be found in \cite{cannarsa2015wavecontrol} or in \cite{zhang2017null}. The first result is concerning the weak solution for \eqref{pb1} and it was established in \cite{cannarsa2015wavecontrol} using a semigroup approach.

\begin{prop}\label{well-pos}
	Given  $f\in L^1(0,T;L^2(0,1))$ and $(u_0,u_1)\in H^1_\alpha\times L^2(0,1)$, there exists a unique weak solution $u\in C^0([0,T];H^1_\alpha)\cap C^1([0,T];L^2(0,1))$ of \eqref{pb1}.	In addition, there exists a positive constant $C_{T,\alpha}$ such that
	\begin{equation}\label{ineq1}
	\sup_{t\in[0,T]}\left( \n{u_t(t)}{L^2(0,1)}^2+\n{u(t)}{H_\alpha^1}^2 \right) 
	\leq C_{T,\alpha}\left(\n{f}{L^1(0,T;L^2(0,1))}^2+\n{u_0}{H_\alpha^1}^2 +\n{u_1}{L^2(0,1)}^2\right).
	\end{equation}
\end{prop}


Associated to \eqref{pb1}, we have the following energy functional
\begin{equation*}
    E(t):=\frac{1}{2}\into (|u_t(t,x)|^2+x^\alpha |u_x(t,x)|^2) \,dx,
\end{equation*}
where $t\in (0,T)$. The previous result establishes that $$E(t)\leq C[\n{f}{L^1(0,T;L^2(0,1))}^2+E(0)],$$ a expected fact for hiperbolic equations. The next result is known by ``hidden regularity", another inherited property from the hiperbolic equations. Like the previous one, this results was also discussed in \cite{cannarsa2015wavecontrol} and \cite{zhang2017null}.

\begin{prop}\label{prop-trace}
    For any weak solution $u$ of \eqref{pb1}, we have $u_x(\cdot,1)\in L^2(0,T)$ and 
    \begin{equation}
        \int_0^T|u_x(t,1)|^2\,dt\leq C\left(\|f\|_{L^1(0,T;L^2(0,1))}^2+E(0)\right)
    \end{equation}
\end{prop}
\begin{proof}
See \cite[Proposition 2.5]{zhang2017null}.
\end{proof}

Now we will discuss the very weak solution for \eqref{pb1}. The way to obtain a well posedness result with less regular data $(z^0,z^1)\in L^2(0,1)\times H^{-1}_\alpha$ is very similar to that used in \cite{cannarsa2015wavecontrol}. Despite that, this result was not analyzed in the papers that we read. So, we will give just a sketch of the proof. First of all, for $u\in H_\alpha^1$ we define $-(x^\alpha u_x)_x\in H_\alpha^{-1}$ by 
\[\langle-(x^\alpha u_x)_x,v\rangle_{H_\alpha^{-1}}=\int_0^1x^\alpha u_xv_x\,dx \ \ \forall v\in H_\alpha^1.\]
Then we define the Hilbert space $Y=L^2(0,1)\times H_\alpha^{-1}$ and the operator $B:D(B)\longrightarrow Y$ given by $$B(u,v)=(-v,-(x^\alpha u_x)_x),$$ where $D(B)=H_\alpha^1\times L^2(0,1)\subset Y$. It is not difficult to see that $(B(U), U)_Y=0 \ \forall U\in D(B)$. In particular, $B$ is a accretive operator. Furthermore, from Lax Milgram Theorem we can deduce that $B$ is M-accretive. It follows that $B$ is skew-adjoint. From semigroup theory we deduce that $B$ is a generator of a semigroup of contractions and this lead us to the following well posedness result:

\begin{prop}\label{well-pos-trans}
	Given  $g\in L^1(0,T;L^2(0,1))$ and $(z^0,z^1)\in L^2(0,1)\times H^{-1}_\alpha$, there exists a unique solution by transposition  $z\in C^0([0,T];L^2(0,1))\cap C^1([0,T];H^{-1}_\alpha)$ of \eqref{pb2}.	In addition, there exists a positive constant $C_{T,\alpha}$ such that
	\begin{equation}\label{ineq-2.3}
	\sup_{t\in[0,T]}\left( \n{z(t)}{L^2(0,1)}^2+\n{z_t(t)}{H_\alpha^{-1}} \right) 
	\leq C_{T,\alpha}\left(\n{g}{L^1(0,T;L^2(0,1))}^2+\n{z^1}{H_\alpha^{-1}}^2 +\n{z^0}{L^2(0,1)}^2\right).
	\end{equation}
\end{prop}

\section{Regularity for weak solutions}\label{sec-weak}

The aim of this section is to prove Theorems \ref{th2.1} and \ref{th2.3}. They bring $L^2$ estimates for weak solutions of system \eqref{pb1} and its derivative, in a $\varepsilon-$neighborhood of the boundary $x=1$.


  We observe that  Theorem \ref{th2.1} will be obtained as a consequence of Theorem \ref{th2.3}. First, let us introduce a suitable function that will be used to prove Theorem \ref{th2.3}.

 Given $\delta>0$ and $\gamma\in(0,\delta)$, for $\kappa=\delta+\gamma$, define $\rho_\kappa:[0,1]\to\mathbb{R}$ by \begin{equation}\label{rhokappa}\rho_\kappa(x)=\begin{cases}
0, & 0\leq x \leq 1-\kappa;\\
\frac{1}{2\delta\gamma}(x-(1-\kappa))^2, & 1-\kappa<x<1-\delta; \\
\frac{1}{\delta}(x-1)+1+\frac{\gamma}{2\delta}, & 1-\delta\leq x\leq 1.
\end{cases}\end{equation}

\begin{figure}[h]
    \centering
   \begin{tikzpicture}
\draw[->] (0,0) -- (6,0) node[below right] {$x$};
\draw[->] (0,0) -- (0,3);
\draw[very thick, green] (0,0) -- (2,0) node[below,black] {$1-\kappa$};
\draw[domain=2:4, smooth, variable=\x, red,very thick]  plot ({\x},{0.2*(\x-2)*(\x-2)}) ;
  \draw[very thick,blue] (4,0.8) -- (5.5,2); 
 \node[below] at (4,0) {$1-\delta$};
\draw[dashed] (4,0) -- (4,0.8) -- (0,0.8) node[left] {$\frac{\gamma}{2\delta}$};
\draw[dashed] (5.5,0) node[below] {$1$} -- (5.5,2) -- (0,2) node[left] {$1+\frac{\gamma}{2\delta}$};
\end{tikzpicture}
    \caption{Graph of $\rho_\kappa$.}
    \label{fig1}
\end{figure}
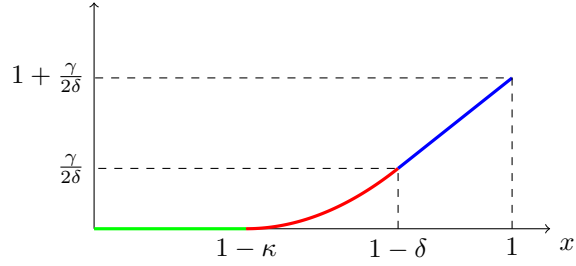
  
Note that $\rho_\kappa\in W^{2,\infty}(0,1)$ and, furthermore
\begin{enumerate}[$(i)$]
    \item $\rho_\kappa$ is a non-decreasing function;
    \item $\rho_\kappa=0$ in $(0,1-\kappa]$, $\rho_\kappa\geq 0$ and $\n{\rho_\kappa'}{L^{\infty}(0,1)}\leq\dps \frac{2}{\kappa}$;
    \item $\rho_\kappa'=\dps\frac{1}{\delta}$ in $(1-\delta,1)$, $\rho_\kappa''=\dps\frac{1}{\delta\gamma}$ in $(1-\ep,1-\delta)$ and $\rho_\kappa''=0$ in $(1-\delta,1)$.
\end{enumerate}






\begin{proof}[Proof of Theorem \ref{th2.3}]

Define 
\[G(\ep):= \frac{1}{\ep}\intw |u_x(t,x)|^2 x^\alpha\dd, \forall \ep\in (0,\ep_0]\]and \[G(0)= \int_0^T|u_x(t,1)|^2\,dt.\]
It is sufficient to prove that $G(\ep)\leq CN_0$, $\forall \ep\in [0,\ep_0]$, where  
\[N_0=\n{f}{L^1(0,T;L^2(0,1))}^2+\n{u_0}{H^1_\alpha}^2+\n{u_1}{L^2(0,1)}^2.\]

Since $G:[0,\ep_0]\longrightarrow \R$ is a continuous function, 
there exists $\delta_0\in [0,\ep_0]$ such that
\[
G(\delta_0)=\max \{ G(\ep); \ep \in [0,\ep_0]\}.
\]

Next, we divide the proof into three situations. \\

\textbf{\underline{Situation 1}: $\delta_0 = 0$.} \\

In this case, we directly apply Proposition \ref{prop-trace} to obtain
\begin{equation*}
    G(\ep)\leq G(0)=\int_0^T|u_x(t,1)|^2\,dt \leq CN_0,\ \forall \ep\in [0,\ep_0].
\end{equation*}

\textbf{\underline{Situation 2}: $\delta_0 \geq \ep _0/2$.} \\

Here, applying Proposition \ref{well-pos}, we get
\begin{align*}
    G(\ep)\leq G(\delta_0)&  =\frac{1}{\delta_0}\int_0^T\int_{1-\delta_0}^1 |u_x(t,x)|^2 x^\alpha\dd
     \leq \frac{2}{\ep_0}\int_0^T\int_{0}^1 |u_x(t,x)|^2 x^\alpha\dd\\
     & \leq \frac{2T}{\ep_0} \sup_{[0,T]}\n{u(t,\cdot)}{H_\alpha^1}^2\leq CN_0.
\end{align*}

\textbf{\underline{Situation 3}: $0<\delta_0<\ep_0/2$.} \\


The proof of this last case is more delicate. For $\delta=\delta_0$, take $\kappa=\delta+\gamma$, where $\gamma\in(0,\delta)$ will be precise later, and $\rho_\kappa$ is given in \eqref{rhokappa}. We observe that\begin{equation}\label{G-0}
    2G(\delta_0)=2\int_0^T\int_{1-\delta_0}^1\frac{1}{\delta_0}|u_x|^2x^\alpha\dd=2\int_0^T\int_{1-\delta_0}^1\rho_\kappa'|u_x|^2x^\alpha\dd\leq  2\int_0^T\int_{1-\kappa}^1\rho_\kappa'|u_x|^2x^\alpha\dd.
\end{equation}
  
We will see that the last integral is bounded by $\frac{3}{2}G(\delta_0)+CN_0$, in order to obtain the desired estimate. This goal will be achieved by the multiplier method.

Firstly, multiply the equation in \eqref{pb1} by $2\rho_\kappa u_x$ and integrating over $Q$. Using integration by parts and recalling that $\rho_\kappa=0$ in $[0,1-\kappa]$,   we have
\begin{multline*}
    -\int_0^T\int_{1-\kappa}^1 2u_tu_{tx}\rho_\kappa\dd+\int_{1-\kappa}^1 2u_t(T)u_x(T)\rho_\kappa\,dx-\int_{1-\kappa}^1 2u_1u_{0x}\rho_\kappa\, dx
    +\int_0^T\int_{1-\kappa}^1 2x^\alpha u_x^2\rho_\kappa'\dd\\+\int_0^T\int_{1-\kappa}^1 2x^\alpha u_xu_{xx}\rho_\kappa \dd -\int_0^T 2u_x^2(t,1)\rho_\kappa(1)\, dt=\intq 2fu_x\rho_\kappa\dd.
\end{multline*}

Since $\rho_\kappa(1-\kappa)=0$ and $u_t(t,1)=0$, we have
\begin{align*}
-\int_0^T\int_{1-\kappa}^1 2u_tu_{tx}\rho_\kappa\dd=& -\int_0^T\int_{1-\kappa}^1 \frac{d}{dx}(u_t^2)\rho_\kappa \dd =\int_0^T\int_{1-\kappa}^1 u_t^2\rho_\kappa'\dd -\int_0^Tu_t^2\rho_\kappa\big\vert_{x=1-\kappa}^1\,dt   \\
& =\int_0^T\int_{1-\kappa}^1 u_t^2\rho_\kappa'\dd.
\end{align*}

Similarly, 
\begin{align*}
\int_0^T\int_{1-\kappa}^1 2x^\alpha u_xu_{xx}\rho_\kappa\dd=-\int_0^T\int_{1-\kappa}^1 x^\alpha u_x^2\rho_\kappa'\dd +\int_0^T u_x^2(t,1)\rho_\kappa(1)\, dt-\int_0^T\int_{1-\kappa}^1 \alpha x^{\alpha-1}u_x^2\rho_\kappa \dd.
\end{align*}
Combining these last three identities, we get
\begin{multline}\label{2.9}
    \int_0^T\int_{1-\kappa}^1 (u_t^2-x^\alpha u_x^2)\rho_\kappa' \dd+\int_0^T\int_{1-\kappa}^1 2x^\alpha u_x^2\rho_\kappa' \dd\\
    =\int_0^T u_x^2(t,1)\rho_\kappa(1)\, dt +\int_0^T\int_{1-\kappa}^1 \alpha x^{\alpha-1} u_x^2 \rho_\kappa\dd+\intq 2fu_x\rho_\kappa\dd \\
    -\int_{1-\kappa}^1 2u_t(T)u_x(T)\rho_\kappa\,dx+\int_{1-\kappa}^1 2u_1u_{0x}\rho_\kappa\, dx.
\end{multline}

Analogously, multiplying the equation in \eqref{pb1} by $\rho_\kappa'u$, we obtain
\begin{multline}\label{2.10}
    -\int_0^T\int_{1-\kappa}^1 (u_t^2-x^\alpha u_x^2)\rho_\kappa' \dd+\int_0^T\int_{1-\kappa}^1 x^\alpha u_xu\rho_\kappa'' \dd\\
    =\intq fu\rho_\kappa'\dd -\int_{1-\kappa}^1 u_t(T)u(T)\rho_\kappa'\,dx+\int_{1-\kappa}^1 u_1u_{0}\rho_\kappa'\, dx.
\end{multline}

Adding \eqref{2.9} and \eqref{2.10}, we get
\begin{multline}\label{2.11}
   \int_0^T\int_{1-\kappa}^1 2x^\alpha u_x^2\rho_\kappa' \dd+\int_0^T\int_{1-\kappa}^1 x^\alpha u_xu\rho_\kappa'' \dd\\
    =\int_0^T u_x^2(t,1)\rho_\kappa(1)\, dt +\int_0^T\int_{1-\kappa}^1 \alpha x^{\alpha-1} u_x^2 \rho_\kappa\dd+\intq 2fu_x\rho_\kappa\dd \\
    -\int_{1-\kappa}^1 2u_t(T)u_x(T)\rho_\kappa\,dx+\int_{1-\kappa}^1 2u_1u_{0x}\rho_\kappa\, dx+\intq fu\rho_\kappa'\dd \\
    -\int_{1-\kappa}^1 u_t(T)u(T)\rho_\kappa'\,dx+\int_{1-\kappa}^1 2u_1u_{0}\rho_\kappa'\, dx.
\end{multline}

Note that the first integral is the one we want to estimate. Thus, let us estimate each integral in the right hand side of this identity by $CN_0$. 

Indeed, Proposition \ref{prop-trace} gives us
\[\int_0^T u_x^2(t,1)\rho_\kappa(1)\, dt \leq CN_0.\]

From inequality \eqref{ineq1},
\begin{equation*}
    \int_0^T\int_{1-\kappa}^1 \alpha x^{\alpha-1} u_x^2 \rho_\kappa\dd \leq \frac{\alpha \n{\rho_\kappa}{L^\infty(0,1)}}{1-\ep_0}\int_0^T\int_{1-\kappa}^1 u_x^2x^\alpha \dd\leq C \sup_{[0,T]}\n{u}{H_\alpha^1}^2 \leq CN_0.
\end{equation*}
\begin{align*}
    \int_0^T\int_{1-\kappa}^1|2f u_x\rho_\kappa| \dd & \leq \frac{2\n{\rho_\kappa}{L^\infty(0,1)}}{(1-\ep_0)^{\alpha/2}}\intw |f||u_x|x^{\alpha/2} \dd\\
    &\leq C \int_0^T \n{f(t)}{L^2(0,1)} \n{x^{\alpha/2}u_x(t)}{L^2(0,1)}\, dt \\
    & \leq C\sup_{[0,T]}\n{u(t)}{H_\alpha^1}\n{f(t)}{L^1(0,T;L^2(0,1))}\\
    &\leq CN_0.
\end{align*}
Similarly, we can estimate 
\[\int_{1-\kappa}^1 2u_t(T)u_x(T)\rho_\kappa\,dx \text{ and }\int_{1-\kappa}^1 2u_1u_{0x}\rho_\kappa\, dx.\]

Before estimate the last three integrals, lets us prove that
\begin{equation}\label{ineq-energy1}
    \frac{1}{\ep^2}\int_{1-\ep}^1 u^2(t,x)\, dx \leq \frac{E(t)}{2(1-\ep_0)^\alpha},\ \forall t\in [0,T],\ \forall \ep\in(0,\ep_0).
\end{equation}
In fact, since $u(t,1)=0$, for any $x\in [1-\ep_0,1)$, we have
\[-u(t,x)=u(t,1)-u(t,x)=\int_x^1 \frac{d}{ds}u(t,s)\, ds\]
From H\"older's inequality, we have
\begin{equation}\label{TFC}
    |u(t,x)|^2\leq (1-x)\int_x^1 u_x^2(t,s)\, ds\leq \frac{1-x}{x^\alpha}\int_x^1 u_x^2(t,s)s^\alpha\, ds.
\end{equation}
Thus,
\begin{align*}
    \frac{1}{\ep^2}\int_{1-\ep}^1| u(t,x)|^2 dx& \leq  \frac{1}{\ep^2}\int_{1-\ep}^1\frac{1-x}{x^\alpha}\left(
    \int_x^1 u_x^2(t,s)s^\alpha\, ds\right)dx\\
    & \leq  \frac{1}{(1-\ep_0)^\alpha\ep^2}\int_{1-\ep}^1 (1-x)\left(\int_0^1 u_x^2(t,s)s^\alpha\, ds\right) dx\\
    & \leq  \frac{E(t)}{(1-\ep_0)^\alpha\ep^2}\int_{1-\ep}^1 (1-x)\,dx\\
     & \leq  \frac{E(t)}{2(1-\ep_0)^\alpha}.
\end{align*}

Now, let us estimate the remain integrals. From \eqref{ineq-energy1} and  H\"older's inequality,
\begin{align*}
    \intq |fu||\rho_\kappa'|\dd& \leq \frac{C}{\kappa}\int_0^T\int_{1-\kappa}^1 |f||u|\dd \leq C\int_0^T \n{f(t)}{L^2(0,1)}\left(\frac{1}{\kappa^2} \int_{1-\kappa}^1u_x^2(t,x)\,dx\right)^{1/2}dt\\
    & \leq C\int_0^T\n{f(t)}{L^2(0,1)} E(t)^{1/2}\,dt\\
    & \leq C\sup_{[0,T]}E(t)^{1/2}\n{f}{L^1(0,T;L^2(0,1))}\leq CN_0.
\end{align*}
Similarly, we can finally estimate 
\[\int_{1-\kappa}^1 u_t(T)u(T)\rho_\kappa'\,dx \text{ and }\int_{1-\kappa}^1 2u_1u_{0}\rho_\kappa'\, dx.\]

Hence, from \eqref{2.11}, we have
\begin{equation}\label{2.13}
   2\int_0^T\int_{1-\kappa}^1 x^\alpha u_x^2\rho_\kappa' \dd\leq \int_0^T\int_{1-\kappa}^1 x^\alpha |u_x||u||\rho_\kappa'' |\dd +CN_0, 
\end{equation}
Recalling \eqref{G-0}, it suffices to prove that 
\begin{equation}\label{ineq4}
\int_0^T\int_{1-\kappa}^1 x^\alpha |u_x||u||\rho_\kappa'' |\dd \leq \frac{3}{2}G(\delta_0)+CN_0.    
\end{equation}

Indeed, since $\rho_\kappa''=0$ in $(1-\delta_0,1)$ and $\rho_\kappa''\leq \frac{1}{\gamma\delta_0}$ in $(1-\kappa,1-\delta_0)$, we have 
\begin{equation}
\begin{split}\label{ineq2}
        \int_0^T\int_{1-\kappa}^1 x^\alpha |u_x||u||\rho_\kappa'' |\dd & \leq \frac{1}{\gamma\delta_0} \int_0^T\int_{1-\kappa}^{1-\delta_0} x^\alpha |u_x||u|\dd\\
    & \leq \left(\int_0^T\int_{1-\kappa}^{1-\delta_0}\frac{1}{\gamma}  |u_x|^2x^\alpha\dd\right)^{1/2}\left(\int_0^T\int_{1-\kappa}^{1-\delta_0}\frac{1}{\gamma  {\delta_0^2}}  |u|^2x^\alpha\dd\right)^{1/2}\\
    & = I_1^{\frac{1}{2}}\,I_2^{\frac{1}{2}}.
\end{split}
\end{equation}

Since $\kappa=\gamma+\delta_0$, we have
\begin{equation*}\label{ineq3}
    I_1 = \frac{1}{\gamma}\left(\kappa G(\kappa)-\delta_0G(\delta_0)\right)\leq G(\delta_0).
\end{equation*}





In order to estimate $I_2$, proceeding as in the proof of \eqref{ineq-energy1}, we can see that
\begin{equation}\label{ineq-I2}
|u|^2x^\alpha\leq C \left((1-x)\int_x^1 u_x^2 s^\alpha\, ds+(1-x)x^{  {\alpha-2}}\int_x^1 u^2 \, ds\right).    
\end{equation}

Besides that, we will need an estimate that the proof is given in the Appendix. Namely, taking $a=1-\ep_0$ in inequality \eqref{ineq-holder}, we have that
\[\sup_{x\in [1-\ep_0,1]}|u(t,x)|\leq \max\left\{\frac{1}{\sqrt{\ep_0}},\sqrt{\frac{\ep_0}{(1-\ep_0)^{\alpha}}}\right\}\n{u(t,\cdot)}{H^1_\alpha}.\]

Hence,
\begin{align*}
    I_2& \leq \frac{C}{\gamma   {\delta_0^2}}\int_0^T\int_{1-\kappa}^{1-\delta_0} (1-x)\int_x^1 u_x^2 s^\alpha\, ds \dd+  \frac{C}{\gamma   {\delta_0^2}}\int_0^T\int_{1-\kappa}^{1-\delta_0}(1-x)x^{\alpha-2}\int_x^1 u^2 \, ds \dd\\
    & \leq \frac{C}{\gamma   {\delta_0^2}}\int_{1-\kappa}^{1-\delta_0} (1-x)^2G(1-x) \dd+ \frac{C}{\gamma   {\delta_0^2}}\int_{1-\kappa}^{1-\delta_0}(1-x)(1-\ep_0)^{2-\alpha}\int_0^T\int_x^1 C_{\alpha,\ep_0}\nh{u(t,\cdot)}^2 \, dsdt\, dx\\
    & \leq \frac{C}{\gamma   {\delta_0^2}}G(\delta_0)\int_{1-\kappa}^{1-\delta_0} (1-x)^2\, dx+ \frac{C_{\alpha,\ep_0}}{\gamma   {\delta_0^2}}\sup_{[0,T]}\nh{u(t,\cdot)}^2\int_{1-\kappa}^{1-\delta_0}(1-x)^2\, dx\\
    & \leq C\frac{9}{4}G(\delta_0)+CN_0.
\end{align*}

Therefore, substituting $I_1$ and $I_2$ into \eqref{ineq2}, we finally get
\begin{equation*}
        \int_0^T\int_{1-\kappa}^1 x^\alpha |u_x||u||\rho_\kappa'' |\dd \leq \left(\frac{9}{4}G(\delta_0)^2+CN_0G(\delta_0)\right)^\frac{1}{2}\leq \frac{3}{2}G(\delta_0)+CN_0,
\end{equation*}
as required in \eqref{ineq4}.
\end{proof}

\begin{proof}[Proof of Theorem \ref{th2.1}]
Let us prove that Theorem \ref{th2.1} as a consequence of Theorem \ref{th2.3}. Indeed, from \eqref{TFC}
\begin{align*}
    \frac{1}{\ep^3}\intw |u(t,x)|^2 \dd & = \frac{1}{\ep^3}\intw \frac{1-x}{x^\alpha}\int_x^1 u_x^2(t,s)s^\alpha\, ds \dd\\
   & \leq \frac{1}{(1-\ep_0)^\alpha\ep^3}\intw (1-x)\int_x^1 u_x^2(t,s)s^\alpha\, ds \dd\\
   & \leq \frac{1}{(1-\ep_0)^\alpha\ep^2}\int_{1-\ep}^1 (1-x)\frac{1}{\ep}\int_0^T\int_{1-\ep}^1 u_x^2(t,s)s^\alpha\, dsdt\, dx \\
   & \leq \frac{CN_0}{(1-\ep_0)^\alpha\ep^2}\int_{1-\ep}^1 (1-x)\, dx \\
   & \leq CN_0.
\end{align*}

\end{proof}

\section{Regularity for very weak solutions}\label{sec-very-weak}

We start this section recalling that the concept of weak solutions for the wave equations, with initial data $(f,u_0, u_1)$ in $L^1 (0,T;L^2 (0,1)) \times H^1_{\alpha} \times L^2 (0,1)$, is placed in Definition \ref{weak}. So, Theorems \ref{th2.1} and \ref{th2.3} bring regularity results for this first type of solution. On the other hand, Definition \ref{trans} is about very weak solutions for wave equations, when the initial data belongs to $L^1 (0,T;L^2 (0,1)) \times L^2 (0,1) \times H^{-1}_{\alpha}$. This current section is devoted to the proof of Theorem \ref{th3.1}, which provides some information about the regularity of this second class of solutions.

Before achieving the main goal of this part, let us describe a useful procedure which allows us to get some additional regularity for solutions of \eqref{pb2}. Take a very weak solution $z$ of \eqref{pb2} related to the initial data $(g,z^0,z^1)\in  L^1(0,T;L^2(0,1))\times L^2(0,1)\times H^{-1}_\alpha$ and let  $\psi^0$ be the solution for the elliptic problem
\[(x^\alpha\psi^0_x)_x=z^1,\ \ \psi^0\in H^1_\alpha.\]
We can see that 
\[\psi(t,x)=\int_0^t z(s,x)\, ds+\psi^0(x)\]
is a weak solution of \eqref{pb2} with initial data $(G,\psi^0,z^0)\in  W^{1,1}(0,T;L^2(0,1))\times H^1_\alpha(0,1)\times L^2(0,1)$, where $G(t,x)=\int_0^t g(s,x)\, ds$.  In this case, Proposition \ref{well-pos} assures that 
\[
\displaystyle \psi\in C^0([0,T];H^1_\alpha) \cap C^1([0,T];L^2(0,1)).
\]
As a consequence, from Proposition \ref{prop-trace}, $\psi_x(\cdot, 1)\in L^2(0,T)$, whence 
\[
z_x(\cdot,1)=\psi_{xt}(\cdot,1)\in H^{-1}(0,T).
\]
We notice that, given $a>0$, $\psi \in C^0([0,T];H^1(a,1))\hookrightarrow  L^2(0,T;H^1(a,1))\equiv H^1(a,1;L^2(0,T))$ and so that
\[
\displaystyle z=\psi_{t}\in H^1(a,1;H^{-1}(0,T))\hookrightarrow C^0([a,1];H^{-1}(0,T)).
\]

 Next, we prove the main result desired here.

\begin{proof}[Proof of Theorem \ref{th3.1}]
As described at the beginning of this section, starting from the very weak solutions $\varphi_\ep,\varphi$ of \eqref{pb2}, we can obtain the related weak solutions
\[
\Phi_{\varepsilon}, \Phi \in C^0([0,T];H_{\alpha}^{1}) \cap C^{1} ([0,T]; L^2 (0,1)).
\]
satisfying
\begin{equation*}
    \begin{cases}
    (x^\alpha \Phi_{\varepsilon x}^{0})_{x} = \varphi_{\varepsilon}^{1}, \text{ where } \Phi_{\varepsilon}^{0} \in H_{\alpha}^{1},\\
    \displaystyle \Phi_{\varepsilon} (t,x) = \int_{0}^{t} \varphi_{\varepsilon} (s,x)ds + \Phi_{\varepsilon}^{0} (x),
    \end{cases}
\end{equation*}
and
\begin{equation*}
    \begin{cases}
    (x^\alpha \Phi_{x}^{0})_{x} = \varphi^{1}$, where $\Phi^{0} \in H_{\alpha}^{1},\\
    \displaystyle \Phi (t,x) = \int_{0}^{t} \varphi (s,x)ds + \Phi ^{0} (x),
    \end{cases}
\end{equation*}
respectively. To be more precise, $\Phi_{\varepsilon}$ comes from the initial data $(H_\ep,\Phi_\ep^0,\varphi_\ep^0)$, where $H_\ep(t,x)=\int_0^t h_\ep(s,x)\, ds$, while $\Phi$ is associated to $(H_,\Phi^0,\varphi^0)$, with $H(t,x)=\int_0^t h(s,x)\, ds$.



Additionally,
\[\Phi_{\ep x}(\cdot,1),\Phi_{ x}(\cdot,1)\in L^2(0,T) \text{ and } \Phi_{\ep x}(\cdot,1) \rightharpoonup \Phi_x(\cdot,1) \text{ in } L^2(0,T),  \]
following
\[\varphi_{\ep x}(\cdot,1),\varphi_{x}(\cdot,1)\in H^{-1}(0,T) \text{ and } \varphi_{\ep x}(\cdot,1) \rightharpoonup \varphi_x(\cdot,1) \text{ in } H^{-1}(0,T).  \]
In order to get $\varphi_x(\cdot,1)\in L^2(0,T)$, it suffices to prove the next claim.

\textbf{\underline{Claim}:} There exists $C>0$ such that 
\begin{equation}\label{bound}
    |\langle\varphi_{x} (\cdot,1), u\rangle|\leq C\|u\|_{L^2 (0,T)} \emph{ for any } u\in \mathcal D (0,T).
\end{equation}

In fact, let us take $u\in \mathcal D (0,T)$ and consider $w(t,x)=(1-x)u(t)$. We notice that $w(t,1)=0$, $w_x (t,1)=-u(t)$ and
\[
\displaystyle \mbox{supp } w(\cdot ,x) = \mbox{supp}(u)
\]
for each $x\in (1-\varepsilon ,1)$. In this case,

\begin{align*}
    \displaystyle \frac{1}{\varepsilon ^3} \intw |w(t,x)|^2 dxdt
    &=\frac{1}{\varepsilon ^3} \int_{0}^{T} |u(t)|^2 \int_{1-\varepsilon}^{1} |x-1|^2 dx dt\\
    &=\frac{1}{3} \int_{0}^{T} |u(t)|^2 dt,
\end{align*}
which means that 
\begin{equation}\label{3.12}
    \bigg( \frac{1}{\varepsilon ^3} \intw |w(t,x)|^2 dxdt\bigg)^{1/2}
    =\frac{1}{\sqrt{3}} \|u\|_{L^2 (0,T)}.
\end{equation}
Since
\[
\displaystyle \Phi_{\varepsilon x} (t,x) = \int_{0}^{t} \varphi_{\varepsilon x} (s,x)ds + \Phi_{\varepsilon x}^{0} (x),
\]
we have $\Phi_{\varepsilon x t} (t,1) = \varphi_{\varepsilon x} (t,1)$ and so that

\begin{align*}
    \displaystyle \langle\varphi_{\varepsilon x} (\cdot,1),u\rangle
&=\langle\Phi_{\varepsilon x t} (\cdot,1),u\rangle 
=-\langle\Phi_{\varepsilon x} (\cdot,1),u'\rangle 
=-\int_{0}^{T} \Phi_{\varepsilon x} (t,1)u'(t) dt \\
&=\begin{multlined}[t]
\frac{3}{\varepsilon ^3} \intw \Phi_{\varepsilon} (t,x) w_t (t,x) dxdt \\
-\frac{3}{\varepsilon ^3} \intw \Phi_{\varepsilon} (t,x) w_t (t,x) dxdt
-\int_{0}^{T} \Phi_{\varepsilon x} (t,1)u'(t) dt
\end{multlined}\\
&=\begin{multlined}[t]
-\frac{3}{\varepsilon ^3} \intw \varphi_{\varepsilon} (t,x) w(t,x) dxdt \\
-\frac{3}{\varepsilon ^3} \intw \Phi_{\varepsilon} (t,x) w_t (t,x) dxdt
-\int_{0}^{T} \Phi_{\varepsilon x} (t,1)u'(t) dt 
\end{multlined}\\
&=: A_{\varepsilon} + B_{\varepsilon},
\end{align*}
where 
\[
A_{\varepsilon} = -\frac{3}{\varepsilon ^3} \intw \varphi_{\varepsilon} (t,x) w(t,x) dxdt
\]
and
\[
B_{\varepsilon} = -\frac{3}{\varepsilon ^3} \intw \Phi_{\varepsilon} (t,x) w_t (t,x) dxdt
    -\int_{0}^{T} \Phi_{\varepsilon x} (t,1)u'(t) dt
\]
From \eqref{3.4} and \eqref{3.12}, it is very clear that

\begin{equation}\label{Aep}
\displaystyle |A_{\varepsilon}| \leq 3 \bigg( \frac{1}{\varepsilon ^3} \intw |\varphi_{\varepsilon} (t,x)|^2 dxdt\bigg)^{1/2} \bigg( \frac{1}{\varepsilon ^3} \intw |w(t,x)|^2 dxdt\bigg)^{1/2}
\leq (3C)^{1/2} \|u\|_{L^2 (0,T)},
\end{equation}
where $C>0$ does not depend on $\varepsilon$. The remainder of this proof is devoted to check that $B_{\varepsilon} \to 0$ as $\varepsilon \to 0$. In fact, using
\begin{itemize}
    \item $\Phi_{\varepsilon} (t,x) w_t (t,x) = \Phi_{\varepsilon} (t,x) (1-x) u'(t)$;
    \item $\displaystyle \int_{x}^{1} \Phi_{\varepsilon x} (t,\xi) d\xi = \Phi_{\varepsilon} (t,1)-\Phi_{\varepsilon} (t,x) =-\Phi_{\varepsilon} (t,x)$;
    \item $\displaystyle \Phi_{\varepsilon x} (t,1) = \frac{3}{\varepsilon ^3 } \int_{1-\varepsilon}^{1} (1-x) \int_{x}^{1} \Phi_{\varepsilon x} (t,1)d\xi dx$,
\end{itemize}
we certainly get
\begin{align*}
    \displaystyle B_{\varepsilon}
    &=-\frac{3}{\varepsilon ^3} \int_{0}^{T} u'(t) \int_{1-\varepsilon}^{1} \Phi_{\varepsilon} (t,x)(1-x)dxdt
    -\int_{0}^{T} \Phi_{\varepsilon x} (t,1)u'(t) dt \\
    &=-\frac{3}{\varepsilon^3} \int_{0}^{T} u'(t) \int_{1-\varepsilon}^{1} \bigg[ \Phi_{\varepsilon} (t,x)(1-x) +(1-x) \int_{x}^{1} \Phi_{\varepsilon x} (t,1) d\xi \bigg] dxdt \\
    &=-\frac{3}{\varepsilon^3} \int_{0}^{T} u'(t) \int_{1-\varepsilon}^{1} (1-x)^2 \bigg[ \frac{1}{1-x} \int_{x}^{1} -\Phi_{\varepsilon x} (t,\xi) d\xi + \Phi_{\varepsilon x} (t,1) \bigg] dxdt \\
    &=\frac{3}{\varepsilon^3} \int_{1-\varepsilon}^{1} (1-x)^2 \bigg[ \frac{1}{1-x} \int_{x}^{1} \langle\Phi_{\varepsilon x} (\cdot, \xi) - \Phi_{\varepsilon x} (\cdot ,1),u' \rangle_{H^{-1}(0,T),H_0^1(0,T)} d\xi \bigg] dx.
\end{align*}
It suffices to prove that $\|\Phi_{\varepsilon x} (\cdot, \xi)\|_{H^{-1}(0,T)}$ is uniformly continuous. Indeed, let $\Psi_{\varepsilon}$ be the weak solution of \eqref{pb2}, with initial data $(\mathbf{H}_\ep,\Psi_\ep^0,\Phi_\ep^0)$, where  
\begin{equation*}
    \begin{cases}
    (x^\alpha \Psi_{\varepsilon x}^{0})_{x} = \varphi_{\varepsilon}^{0}, \text{ where } \Psi_{\varepsilon}^{0} \in H_{\alpha}^{1}\cap H^2_\alpha,\\
    \displaystyle \Psi_{\varepsilon} (t,x) = \int_{0}^{t} \Phi_{\varepsilon} (s,x)ds + \Psi_{\varepsilon}^{0} (x),
    \end{cases}
\end{equation*}
and $\mathbf{H}_\ep(t,x)=\int_0^t h_\ep(s,x)\, ds$. Hence, $\Psi_\ep\in L^2(0,T;H^1_\alpha\cap H^2_\alpha)$ and, in particular, $\Psi_\ep \in L^2(0,T;H^2(1-\ep,1))$. As a result, 
$\Psi_{\ep x}\in L^2(0,T;H^1(1-\ep,1))\equiv H^1(1-\ep,1;L^2(0,T))$, whence
\[\Phi_{\ep x}=\Psi_{\ep x t}\in H^{1}(1-\ep,1;H^{-1}(0,T))\hookrightarrow C^0([1-\ep,1];H^{-1}(0,T)),\]
following the required uniform continuity.


Finally, \eqref{liminf} is an immediate consequence of
$\langle\varphi_{\varepsilon x} (\cdot,1),u\rangle = A_{\varepsilon} + B_{\varepsilon},$ \eqref{Aep} and $B_\ep \to 0$.
\end{proof}

\appendix
\renewcommand{\thesection}{\Alph{section}} 

\section{Appendix}

The aim of this section is to obtain the inequality \eqref{ineq-holder}, that was fundamental in the proof of Theorem \ref{th2.3}. Besides that, we end up proving an embedding for the Sobolev space $H_\alpha^1$ in a H\"older continuous space.

Given $a\in (0,1)$, let us denote by $H^1(a,1)$ and $C^{0,1/2}([a,1])$  classical Sobolev and H\"older spaces.

\begin{lem}\label{im-ha-h1} $\ha \hookrightarrow H^1(a,1),\ \forall a\in (0,1)$ and
\[\n{u}{H^1(a,1)}\leq C_{\alpha,a}\nh{u}, \]
where $C_{\alpha,a }=\sqrt{\max\{1,1/a^\alpha\}}$.
\end{lem}

\begin{proof}
Given $a\in (0,1)$, if $u\in \ha$, then $u\in L^2(a,1)$. Besides that, we have
\begin{equation}\label{eq1}
\n{u'}{L^2(a,1)}^2=\int_a^1 x^{-\alpha}x^\alpha |u'|^2\, dx\leq \frac{1}{a^\alpha}|u'|^2\, dx =\frac{1}{a^\alpha}\n{x^{\alpha/2}u'}{L^2(0,1)}^2.
\end{equation}

Therefore,
\[\n{u}{H^1(0,1)}^2\leq \max\{1,1/a^\alpha\}\nh{u}^2.\]
\end{proof}

Classical Sobolev Imbedding,  together with Lemma \ref{im-ha-h1}, give us that 
\begin{equation}\label{imb-ha-c0}
\ha \hookrightarrow H^1(a,1)\hookrightarrow C^{0,1/2}([a,1]), \text{ for any }  a\in (0,1).
\end{equation}
Moreover,  we have that any $u\in \ha$ satisfies
\[u(x)-u(y)=\int_y^x u'(\xi)\, d\xi,\ \forall x,y\in (0,1).\]

However, let us give the proof of \eqref{imb-ha-c0}  in order to see how the constant of embedding depends on $a,\alpha$.

\begin{lem}
$\ha \hookrightarrow C^{0,1/2}([a,1])$,  for any   $a\in (0,1)$, and 
\begin{equation}\label{ineq-holder}
\n{u}{C^{0,1/2}([a,1])}\leq C_{\alpha,a}\nh{u},\ \forall u\in \ha,  
\end{equation}
where $C_{\alpha,a}:=\max\left\{(1 -a)^{-1/2},\frac{(1-a)^{1/2}}{a^{\alpha/2}}\right\}$.
\end{lem}

\begin{proof}
Given $x,y\in [a,1]$, from Cauchy-Schwartz inequality and \eqref{eq1}, we have that 
\begin{equation*}
\begin{split}
|u(x)-u(y)|& \leq \int_y^x |u'|\, d\xi \leq \left(\int_y^x\, d\xi\right)^{1/2} \left(\int_y^x |u'|^2\ d\xi\right)^{1/2}\\
&=|x-y|^{1/2}\left(\int_a^1 |u'|^2\ d\xi\right)^{1/2}\leq|x-y|^{1/2} \frac{1}{a^{\alpha/2}}\nh{u}. 
\end{split}
\end{equation*}
Hence,
\[\left[u\right]_{C^{0,1/2}([a,1])}\leq \frac{1}{a^{\alpha/2}}\nh{u}. \]

Besides that, 
\begin{equation*}
\begin{split}
|u(y)|& \leq |u(x)|+|u(x)-u(y)|\leq |u(x)|+|x-y|^{1/2} \frac{1}{a^{\alpha/2}}\nh{u}\\
& \leq |u(x)|+ \frac{(1-a)^{1/2}}{a^{\alpha/2}}\nh{u},\ \forall x,y\in [a,1].
\end{split}
\end{equation*}
Integrating with respect to $x$ over $[a,1]$, we get
\[(1 -a)|u(y)|\leq (1 -a)^{1/2}\n{u}{L^2(0,1)}+\frac{(1-a)^{3/2}}{a^{\alpha/2}}\nh{u},\]
we obtain
\begin{align*}
|u(y)|& \leq (1 -a)^{-1/2}\n{u}{L^2(0,1)}+\frac{(1-a)^{1/2}}{a^{\alpha/2}}\nh{u}\\
& \leq \max\left\{(1 -a)^{-1/2},\frac{(1-a)^{1/2}}{a^{\alpha/2}}\right\}\nh{u}, \ \forall y\in [a,1],
\end{align*}
which gives us the desire inequality.
\end{proof}
\bibliography{references}
\end{document}